\documentclass[12pt]{amsart}
\usepackage[utf8]{inputenc}
\usepackage{graphicx}
\usepackage[english]{babel}
\usepackage{color}
\usepackage{float}
\usepackage{caption}
\usepackage{subcaption}
\usepackage{mathtools}
\usepackage{amsmath}
\usepackage{makecell}
\usepackage{cancel}
\usepackage{amssymb}
\usepackage{amsthm}
\usepackage{algorithm2e}
\usepackage{mathdots}
\usepackage{enumerate}
\usepackage{xcolor}
\usepackage{setspace}
\usepackage{tcolorbox}
\usepackage[colorlinks=true, linkcolor = blue]{hyperref}
\usepackage{perpage}
\usepackage[margin=1in]{geometry}

\newcommand{\abs}[1]{\left\lvert{#1}\right\rvert}

\newcommand{\ceil}[1]{\left\lceil{#1}\right\rceil}
\newcommand{\pare}[1]{\left({#1}\right)}

\renewcommand{\tilde}{\widetilde}
\renewcommand{\phi}{\varphi}
\renewcommand{\epsilon}{\varepsilon}
\newcommand{\subsum}[1]{\sum_{\substack{#1}}}
\newtheorem{theorem}{Theorem}[section]

\newtheorem{lemma}[theorem]{Lemma}
\newtheorem{proposition}[theorem]{Proposition}
\theoremstyle{definition}

\theoremstyle{definition}
\newtheorem*{remark}{Remark}

\newcommand{\I}[1]{\mathbf{1}_{\textbf{S}(y)}({#1})}
\newcommand{\Card}[1]{\Psi\left({#1}, y\right)}

\newcommand{\M}{\ensuremath{\mathcal{M}}}
\title{Smooth Numbers in Short Intervals}
\author{Sarvagya Jain}
\address{Department of Mathematics and Statistics, University of Turku, Turku, Finland}
\email{sarvagyajain.math@gmail.com}
\date{}
\numberwithin{equation}{section}
\begin{document}
\begin{abstract}
    Let \( X \geq y \geq 2 \), and let \( u = \frac{\log X}{\log y} \). We say a number is \textit{$y$-smooth} if all of its prime factors are less than or equal to \( y \). In this paper, we study the distribution of $y$-smooth numbers in short intervals. In particular, for \( y \geq \exp\left( (\log X)^{2/3 + \epsilon} \right) \), we show that the interval \( [x, x+h] \) contains a $y$-smooth number for almost all \( x \in [X, 2X] \), provided \( h \geq \exp\left( (1 + \epsilon) \left( \frac{11}{8} u \log u + 4 \log \log X \right) \right) \), and \( X \) is sufficiently large depending on \( \epsilon \). This result improves upon an earlier result by Matomäki. Additionally, we provide the corresponding ``all intervals" type result.
\end{abstract}
\maketitle
\vspace{-0.5in}
\section{Introduction}
Let \( P(n) \) denote the largest prime factor of \( n \). A number \( n \) is said to be \( y \)-smooth if \( P(n) \leq y \). We denote the set of all \( y \)-smooth numbers by \( \mathbf{S}(y) \). Let \( \Card{x} \) denote the number of \( y \)-smooth numbers up to \( x \), that is, \( |\mathbf{S}(y) \cap [1, x]| \). It is known that \( \Card{x} \sim \rho(u) x \) holds for a wide range of the smoothness parameter \( y \), where \( u = \frac{\log x}{\log y} \) and \( \rho(u) \) is the Dickman function (see (\ref{dickman defn}) below for its definition). Despite this, our understanding of the distribution of smooth numbers in short intervals is still far from complete.

One would expect a result like \( \Card{x+h} - \Card{x} \sim \rho(u) h \) to hold for a wide range of parameters \( h \) and \( y \). Let \( \psi(x) \to \infty \) as \( x \to \infty \), and let \( u > 0 \) be fixed. In a breakthrough paper, Matomäki and Radziwiłł showed that such a result holds for almost all \( x \in [X, 2X] \), at least when \( y = X^{1/u} \) and \( h = \psi(X) \) (see \cite[Corollary 6]{matomaki-radziwill}). In the same paper, they also showed that, for every \( \epsilon > 0 \), there exists a sufficiently large constant \( C_\epsilon > 0 \) such that every interval \( [x, x+h] \) contains an \( x^\epsilon \)-smooth number, provided \( h \geq C_\epsilon \sqrt{x} \) (see \cite[Corollary 1]{matomaki-radziwill}).

It is natural to ask whether \( \Card{x+h} - \Card{x} \geq 1 \) holds for a wider range of parameters \( h \) and \( y \). Since \( \rho(u) = \frac{1}{u^{u+o(u)}} \), based on probabilistic heuristics, one expects this to hold for almost all short intervals of size \( h \geq \exp\left( (1 + \epsilon) u \log u \right) \). However, this kind of result seems far out of reach.

In this paper, we apply the ideas introduced by Matomäki and Radziwiłł to make progress in this direction.
\begin{theorem}\label{almost all interval}
For any \( \epsilon > 0 \), there exists a positive constant \( C = C(\epsilon) \) such that the following holds. Let \( X \geq 2 \) be large enough depending on \( \epsilon \). If
\[
\exp\left( C (\log X)^{2/3} (\log \log X)^{4/3} \right) \leq y \leq X^{\frac{1}{C}}
\]
and
\[
h \geq \exp\left( (1 + \epsilon) \left( \frac{11}{8} u \log u + 4 \log \log X \right) \right),
\]
for \( u = \frac{\log X}{\log y} \), then the interval \( [x, x+h] \) contains a \( y \)-smooth number for almost all \( x \in [X, 2X] \).
\end{theorem}
In particular, we improve upon an earlier result of Matomäki \cite[Theorem 1.3]{matomaki}. She showed such a result for 
\[
y \geq \exp\left( (\log X)^{2/3} (\log \log X)^{4/3 + \epsilon} \right)
\]
and 
\[
h \geq \exp\left( \left( \frac{14}{3} + \epsilon \right) \left(4u \log u + \log \log X\right) \right).
\]

We also prove a result about the existence of a \( y \)-smooth number in all intervals of the form \( [x, x + x^{1/2 + o(1)}] \) for sufficiently large \( x \).
\begin{theorem}\label{all interval}
For any \( \epsilon > 0 \), there exists a positive constant \( C = C(\epsilon) \) such that the following holds. If $x$ is large in terms of \(\epsilon\),
\[
\exp\left( C (\log x)^{2/3} (\log \log x)^{4/3} \right) \leq y \leq x^{\frac{1}{C}},
\]
and
\[
h \geq \sqrt{x} \exp\left( (1 + \epsilon) \left( \frac{11}{16} \tilde{u} \log \tilde{u} + 2 \log \log x \right) \right),
\]
where \( \tilde{u} = \frac{\log x}{\log y} \), then the interval \( [x, x+h] \) contains a \( y \)-smooth number.
\end{theorem}
The above result can be contrasted with Theorem 1.1 in \cite{matomaki}. Matomäki proved in this theorem that, for sufficiently large \( x \), one could take
\[
h \geq \sqrt{x} \exp\left( \left( \frac{7}{3} + \epsilon \right) \left( 4\tilde{u} \log \tilde{u} + \log \log x \right) \right),
\]
provided that \( y \geq \exp\left( (\log x)^{2/3} (\log \log x)^{4/3 + \epsilon} \right) \).

There have been various other results concerning the distribution of smooth numbers in short intervals. For example, Hildebrand and Tenenbaum \cite{hildebrand-tenenbaum} proved an asymptotic formula for the number of smooth numbers in almost all short intervals. However, the intervals they consider are significantly longer than those considered here. Building on their work, Granville and Friedlander \cite{granville-friedlander} obtained the corresponding ``all intervals" type result. We refer the reader to \cite{granville,hildebrand-tenenbaum} for more detailed surveys of results on smooth numbers and their applications to other problems in analytic number theory and cryptography.
\subsection{Sketch of the Argument} We will give a brief outline of the proof of Theorem \ref{almost all interval} to orient the reader. In this sketch, we use the imprecise notion of inequality $\lesssim$ and $\gtrsim$ to hide the $\rho(u)^\epsilon$ and $\log X$ factors.

Suppose that we want to show for some $h\geq 2$, the intervals $[x, x+h]$ contains a $y$-smooth number for almost all $x\in [X, 2X]$. Choose weights $\{w_n\}_n$ such that $w_n\geq 0$ if $n\in \textbf{S}(y)$ and $0$ otherwise. Let $H$ be such that $\sum_{x\leq n\leq x+H}w_n>0$ for all $x\in [X, 2X]$. It turns out that, for our choice of weights, we can take $H = Xy^{-3/8}$ (see Lemma \ref{weights av} below). If we show that 
$$\left|\frac{1}{h}\sum_{x\leq n\leq x+h}w_{n} - \frac{1}{H}\sum_{x\leq n\leq x+H}w_n\right| = o\pare{\frac{1}{H}\sum_{x\leq n\leq x+H}w_n}$$ for almost all $x\in [X, 2X]$, then we will have our result. By a routine Chebyshev-style argument (see Section \ref{Chebysev style} below), we can reduce it further to showing that 
$$\frac{1}{X}\int_{X}^{2X}\left|\frac{1}{h}\sum_{x\leq n\leq x+h}w_{n} - \frac{1}{H}\sum_{x\leq n\leq x+H}w_n\right|^2dx = o\pare{\pare{\frac{1}{X}\sum_{X\leq n\leq 2X}w_n}^2}.$$
With our choice of \(w_n\), this reduces to essentially showing that the left-hand side is \(o(\rho(u)^2)\). Now, letting $F(s) = \sum_{n}w_n/n^s$, we bound the second moment on the left-hand side by using a Parseval-style bound (see Section \ref{parseval reduction} below). Doing this, we essentially get the bound 
\begin{align*}
    \frac{1}{X}\int_{X}^{2X}\left|\frac{1}{h}\sum_{x\leq n\leq x+h}w_{n} - \frac{1}{H}\sum_{x\leq n\leq x+H}w_n\right|^2dx \lesssim \int_{y^{1/8}}^{\frac{X}{h}}|F(1+it)|^2dt.
\end{align*}
We choose our weights so that $F(s)$ has the form $P_1(s)P_2(s)P_3(s)^JM(s)$ for $P_j(s) = \sum_{p_j\sim P_j}1/p^s$ and $M(s) = \sum_{m\in \M}1/m^s$ for some $\M\subseteq \textbf{S}(y)$ and $P_1\leq P_2\leq P_3\leq y$. If we choose parameters $P_1, P_2, P_3, J$ and the set $\M$ carefully, we may exploit the factorisation of our polynomial to get non-trivial savings for its mean value estimate using the ideas introduced in \cite{matomaki-radziwill}. Our goal is to get enough savings so that we have the bound
\begin{align*}
    \int_{y^{1/8}}^{\frac{X}{h}}|F(1+it)|^2dt = o\pare{\rho(u)^2}.
\end{align*}
The details of how one can get such a saving are given in Section \ref{main content}. Here we content ourselves with a very informal sketch. We begin by choosing $0<\alpha_1<\alpha_2<1/4$ with $\alpha_1, \alpha_2\approx 1/4$ and then partitioning the interval $[y^{1/8}, X/h]$ into $\mathcal{T}_1\cup\mathcal{T}_2\cup\mathcal{T}_3$, where $t\in \mathcal{T}_1$ if $|P_1(1+it)|\leq P_1^{-\alpha_1}$, $t\in\mathcal{T}_2$ if $|P_1(1+it)|>P_1^{-\alpha_1}$, $|P_2(1+it)|\leq P_2^{-\alpha_2}$ and $\mathcal{T}_3$ contains the remaining values of $t$.

To bound the integral over $\mathcal{T}_1$, we note that
$$\int_{\mathcal{T}_1}|F(1+it)|^2 dt\lesssim P_1^{-2\alpha_1}\int_{-X/h}^{X/h}|P_2(1+it)P_{3}(1+it)^JM(1+it)|^2dt.$$
At this point, it may seem tempting to use the standard mean value theorem (see Lemma \ref{MVT} below) to bound the integral above. But the approach mentioned above turns out to be lossy as it does not incorporate the sparsity of the coefficients of the Dirichlet polynomial $P_2(s)P_{3}(s)^JM(s)$, say $\{a_n\}_n$, that comes from being supported on $y$-smooth numbers. To exploit this sparsity better, one instead uses a variant of the standard mean value theorem that preserves some non-diagonal terms (see Lemma \ref{improved mvt} below). This naturally leads us to upper bounding sums essentially of the form $\sum_{n\sim N}\I{n}\I{n+k}$ (see Lemma \ref{contri T1} below). These kinds of sums have been considered in earlier works, and bounds of the form $O(\rho(u)^{\phi}N)$ for some $1<\phi<2$ are known. We use the best existing bound, i.e., the largest available $\phi$ (see Lemma \ref{solns-prod-linear} below). This essentially results in an upper bound of the form
\begin{align*}
    \int_{\mathcal{T}_1}|F(1+it)|^2\lesssim \frac{P_1^{1/2}\rho(u)}{h}+ \frac{\rho(u)^\phi}{P_1^{1/2}}.
\end{align*}

To bound the integral over $\mathcal{T}_2$, we note that
$$\int_{\mathcal{T}_2}|F(1+it)|^2 dt\lesssim P_2^{-2\alpha_2}\int_{-X/h}^{X/h}|P_1(1+it)P_{3}(1+it)^JM(1+it)|^2dt.$$
To bound the integral above, we exploit $|P_1^{\alpha_1}P_1(1+it)|\geq 1$ to increase the length of Dirichlet polynomial $|P_1(1+it)P_{3}(1+it)^JM(1+it)|^2$ by multiplying it with $|P_1^{\alpha_1}P_1(1+it)|^{2\ell}$ for some appropriately chosen $\ell = \ceil{\frac{\log P_2}{\log P_1}}$ (see Lemma \ref{contri T2} below). This results in an upper bound of the form 
$$\int_{\mathcal{T}_2}|F(1+it)|^2 dt\lesssim P_2^{-2(\alpha_2-\alpha_1)}\int_{-X/h}^{X/h}|P_1(1+it)^{\ell+1}P_{3}(1+it)^JM(1+it)|^2dt.$$
The motivation behind these manipulations is the well-known fact that the standard mean value theorem is the most efficient when the length of the Dirichlet polynomial involved and the range of integration are of similar order (see \cite[Chapter 9]{iwaniec-kowalski} for further discussion). Applying the standard mean value theorem and choosing $P_2$ appropriately in terms of $P_1$, we have that
\begin{align*}
    \int_{\mathcal{T}_2}|F(1+it)|^2\lesssim \frac{P_1^{1/2}\rho(u)}{h}+ \frac{\rho(u)^\phi}{P_1^{1/2}}.
\end{align*}
Finally, for bounding the integral over $\mathcal{T}_3$, we begin by noting the bound $|P_3(1+it)| \lesssim P_3^{-\sigma_0}$, where $\sigma_0 \asymp (\log X)^{-2/3}(\log\log X)^{-1/3}$ as $y^{1/8}\leq t\leq X$ for all $t\in \mathcal{T}$. This is nothing but the upper bound that one gets from using Perron's formula and the Vinogradov-Korobov zero-free region (see Lemma \ref{prime-poly-pw-bound} below). To maximise this saving, it is natural to choose $P_3$ as large as possible, i.e., $P_3\approx y$. So, to have a non-trivial saving, we must have $y\geq \exp\pare{\sigma_0^{-1}}$. Further bookkeeping results in a bound for $y$ slightly larger than this. Substituting the bound for $P_3(s)$, we arrive at the bound
\begin{align*}
    \int_{\mathcal{T}_3}|F(1+it)|^2dt\lesssim y^{-2J\sigma_0}\int_{\mathcal{T}_3}|P_1(1+it)P_2(1+it)M(1+it)|^{2}dt.
\end{align*}
Now we can find a well-spaced set $\mathcal{T}\subseteq \mathcal{T}_3$ so that we have the bound
\begin{align*}
    \int_{\mathcal{T}_3}|F(1+it)|^2dt\lesssim y^{-2J\sigma_0}\sum_{t\in\mathcal{T}}|P_1(1+it)P_2(1+it)M(1+it)|^{2}.
\end{align*}
By using Hal\'asz-Montgomery inequality (see Lemma \ref{halasz-montgomery} below), we can bound the above by
\begin{align*}
    \int_{\mathcal{T}_3}|F(1+it)|^2dt\lesssim y^{-2J\sigma_0}\rho(u)\pare{1+\frac{|\mathcal{T}|y^J}{\sqrt{X}}}.
\end{align*}
It turns out that one can use $|P_2(1+it)|>P_2^{-\alpha_2}$ for all $t\in \mathcal{T}$ to show that $|\mathcal{T}|y^J\lesssim X^{1/2-\delta}$ for some $\delta>0$ (using Lemma \ref{prime-poly-large-value} below). This step is the primary reason for choosing $\alpha_2<1/4$. Additionally, by choosing $J$ appropriately, we have the bound $y^{-2J\sigma_0}\lesssim \rho(u)^{24}$. Thus, we have the bound
$$\int_{\mathcal{T}_3}|F(1+it)|^2dt\lesssim \rho(u)^{25}.$$
In conclusion, we have the bound
\begin{align*}
    \int_{y^{1/8}}^{X/h}|F(1+it)|^{2}dt\lesssim \frac{P_1^{1/2}\rho(u)}{h}+ \frac{\rho(u)^\phi}{P_1^{1/2}}+\rho(u)^{25}.
\end{align*}
Since we want the right-hand side to be $o(\rho(u))^2$, we are forced to choose $P_1\gtrsim \rho(u)^{2\phi-4}$. Finally, if we set $h\gtrsim P_1^{1/2}\rho(u)^{-1}\gtrsim \rho(u)^{\phi-3}$, the right-hand side will be $o(\rho(u)^2)$. By our earlier remarks, we will have a $y$-smooth number in almost all short intervals of size $\gtrsim\rho(u)^{\phi-3}$ (the additional $\log X$ factors arise from careful bookkeeping).
\begin{remark}
    As can be seen from the above sketch, improvements in Lemmas \ref{prime-poly-pw-bound} and \ref{solns-prod-linear} are expected to lead to refinements in the smoothness parameter and short interval length, respectively.
\end{remark}
\subsection{Acknowledgements} 
The author would like to thank his advisor, Kaisa Matomäki, for suggesting the problem and for her guidance throughout the project. Her insights not only led to a clearer exposition but also contributed to improving the quality of both the smoothness parameter and the short interval length. The author was supported by the Academy of Finland, Centre of Excellence (Grant No. 346307), and the University of Turku Graduate School Exactus fellowship while working on this project.
\section{Dirichlet Polynomial Estimates}
In this section, we recall various estimates from the literature for Dirichlet polynomials that we will need. Throughout this section, we let \( X \geq 1 \) and \( T > 0 \). Additionally, for \( a_n \in \mathbb{C} \), we define the Dirichlet polynomial
\[
G(s) := \sum_{n \sim X} \frac{a_n}{n^s}.
\]
The following result is a standard mean value estimate that provides bounds for the second moment of a Dirichlet polynomial.

\begin{lemma}\label{MVT}
    One has the estimate
    \[
    \int_{-T}^{T} |G(it)|^2 \, dt \ll (T + X) \sum_{n \sim X} |a_n|^2.
    \]
\end{lemma}

\begin{proof}
    For a proof, see \cite[Theorem 9.1]{iwaniec-kowalski}.
\end{proof}

The mean value estimate from Lemma \ref{MVT} loses some of the information about the sparsity of the Dirichlet polynomial due to the term \( X \sum_n |a_n|^2 \). In our proof, we will use a Dirichlet polynomial whose coefficients are supported on a subset of \( \mathbf{S}(y) \) and are therefore quite sparse. The following variant of the mean value theorem will allow us to exploit the sparsity of the Dirichlet polynomial to some extent. This variant has been used in works on almost primes in short intervals \cite{matomaki-teravainen, teravainen}.

\begin{lemma}\label{improved mvt}
    One has the estimate
    \[
    \int_{-T}^{T} |G(it)|^2 \, dt \ll T \sum_{n \sim X} |a_n|^2 + T \sum_{1 \leq k \leq \frac{2X}{T}} \sum_{n \sim X} |a_n| |a_{n+k}|.
    \]
\end{lemma}

\begin{proof}
    For a proof, see \cite[Lemma 7.1]{iwaniec-kowalski}. See also \cite[Lemma 4]{teravainen}.
\end{proof}

Later, we will also amplify the length of a Dirichlet polynomial by introducing a moment of a prime-supported polynomial. The following lemma will allow us to control the mean value with these augmentations.

\begin{lemma}\label{moment computation}
    Let \( X \geq P_2\geq  P_1 \geq 2 \) be parameters and set \( \ell = \lceil \frac{\log P_2}{\log P_1} \rceil \). For
    \[
    P(s) = \sum_{p \sim P_1} \frac{1}{p^s} \quad \text{and} \quad A(s) = \sum_{n \sim \frac{X}{P_2}} \frac{a_n}{n^s},
    \]
    we have
    \[
    \int_{-T}^T |P(1 + it)^\ell A(1 + it)|^2 \, dt \ll \left( \frac{T}{X} + 2^\ell P_1 \right) (\ell + 1)!^2 \max_n |a_n|^2.
    \]
\end{lemma}

\begin{proof}
    This follows immediately from \cite[Lemma 13]{matomaki-radziwill}.
\end{proof}

A set \( \mathcal{T} \subseteq \mathbb{R} \) is said to be \textit{well-spaced} if for all \( t, u \in \mathcal{T} \) with \( t \neq u \), we have \( |t - u| \geq 1 \). At some point in our argument, we will encounter a prime-supported Dirichlet polynomial that is large on a well-spaced set. The following lemma will allow us to use this information to obtain a power-saving bound for the size of the well-spaced set.

\begin{lemma}\label{prime-poly-large-value}
    Let 
    \[
    P(s) = \sum_{p \sim P} \frac{a_p}{p^s} \quad \text{with } |a_p| \leq 1.
    \]
    Let \( \mathcal{T} \subset [-T, T] \) be a sequence of well-spaced points such that \( |P(1 + it)| \geq V^{-1} \) for every \( t \in \mathcal{T} \). Then
    \[
    |\mathcal{T}| \ll T^{2 \frac{\log V}{\log P}} V^2 \exp\left( 2 \frac{\log T}{\log P} \log \log T \right).
    \]
\end{lemma}

\begin{proof}
    For a proof, see \cite[Lemma 8]{matomaki-radziwill}.
\end{proof}

Having produced a relatively small well-spaced set, the following discrete mean value theorem, also known as the Halász-Montgomery inequality, will allow us to upper bound the mean square of a Dirichlet polynomial over such sets.

\begin{lemma}\label{halasz-montgomery} 
    Let \( \mathcal{T} \subset [-T, T] \) be a well-spaced set. Then,
    \[
    \sum_{t \in \mathcal{T}} |G(it)|^2 \ll (X + |\mathcal{T}| T^{\frac{1}{2}})(\log T) \sum_{n \sim X} |a_n|^2.
    \]
\end{lemma}

\begin{proof}
    For a proof, see \cite[Theorem 9.6]{iwaniec-kowalski}.
\end{proof}

Finally, we need the following pointwise bound, which is a consequence of Perron's formula and the Vinogradov-Korobov zero-free region. For a complex number \( s \), let \( \sigma := \Re(s) \) and \( \tau := |\Im(s)| + 2 \). Then, for some constant \( A_{vk} > 0 \), the region
\[
    \sigma \geq 1 - \frac{A_{vk}}{(\log \tau)^{2/3} (\log \log \tau)^{1/3}}
\]
denotes the Vinogradov-Korobov zero-free region. Let
\[
    \sigma_0 := \frac{A_{vk}}{(\log X)^{2/3} (\log \log X)^{1/3}}.
\]

\begin{lemma}\label{prime-poly-pw-bound}
    For \( P \geq \exp\left( (\log X)^{2/3} (\log \log X)^{4/3} \right) \), let
    \[
    P(s) = \sum_{p \sim P} \frac{1}{p^s}.
    \]
    Then, for all \( |t| \leq X \),
    \[
    |P(1 + it)| \ll P^{-\sigma_0} + \frac{(\log X)^3}{1 + |t|}.
    \]
\end{lemma}

\begin{proof}
    The result follows from standard techniques used in \cite[Section 1.4]{harman}. The stronger bound above can be derived using the same methods, instead of the weaker one mentioned in \cite[Lemma 1.5]{harman}.
\end{proof}
\section{Some Results About Smooth Numbers}

In this section, we recall results about smooth numbers that will be useful for our analysis. We first state some results concerning the distribution of smooth numbers.

The global distribution of smooth numbers for a wide range of smoothness parameters is governed by the Dickman function \( \rho: \mathbb{R}_{\geq 0} \to \mathbb{R} \). It is defined as the continuous solution to the system
\begin{equation}\label{dickman defn}
    \begin{aligned}
        \rho(u) &= 1 & \text{for} \quad 0 \leq u \leq 1, \\
        -u \rho'(u) &= \rho(u - 1) & \text{for} \quad u > 1.
    \end{aligned}
\end{equation}
As mentioned in the introduction, it is well known that \( \rho(u) = \frac{1}{u^{u + o(u)}} \) (see, for example, \cite[(1.6)]{granville}). We will make repeated use of this estimate throughout our analysis without further reference.

In our analysis, we consider \( y \)-smooth numbers where some of their prime factors lie in certain dyadic ranges. Counting these numbers will involve controlling the ratio \( \rho(u - v) / \rho(u) \) when \( v \) is much smaller than \( u \). To accomplish this, we use the following lemma.

\begin{lemma}\label{dickman-behaviour}
    For \( u > 2 \) and \( |v| \leq u/2 \), we have
    \begin{align*}
        \rho(u - v) &= \rho(u) \exp\left( v \xi(u) + O\left( \frac{1 + v^2}{u} \right) \right),
    \end{align*}
    where \( \xi(u) \) satisfies the asymptotic
    \[
        \xi(u) = \log u + \log \log(u + 2) + O\left( \frac{\log \log(u + 2)}{\log(u + 2)} \right).
    \]
\end{lemma}

\begin{proof}
    See \cite[Corollary 2.4 and Lemma 2.2]{hildebrand-tenenbaum}.    
\end{proof}
The following lemma gives us asymptotics for the number of smooth numbers in long intervals and moderately short intervals.
\begin{lemma}\label{smooth-number-estimates}
    Fix \( \epsilon > 0 \). Let \( X \geq 2 \), \( y \geq \exp\left( (\log \log X)^{5/3 + \epsilon} \right) \), and write \( u = \frac{\log X}{\log y} \).
    \begin{enumerate}[(i)]
        \item We have the estimate
        \[
        \Card{X} = X \rho(u) \left( 1 + O_\epsilon\left( \frac{\log(u + 1)}{\log y} \right) \right).
        \]
        \item For \( X y^{-\frac{5}{12}} \leq h \leq X \), we have
        \[
        \Card{X + h} - \Card{X} = h \rho(u) \left( 1 + O_\epsilon\left( \frac{\log(u + 1)}{\log y} \right) \right).
        \]
    \end{enumerate}
\end{lemma}

\begin{proof}
    For (i) and (ii), see \cite[Theorem 1, Theorem 3]{hildebrand} respectively.
\end{proof}
Next, we state an upper bound for the number of integers \( n \in [X, 2X] \) such that both \( n \) and \( an + b \) lie in \( \mathbf{S}(y) \). This will allow us to bound the off-diagonal contribution in Lemma \ref{improved mvt} when applying it to a Dirichlet polynomial supported on \( [X, 2X] \cap \mathbf{S}(y) \). 
\begin{lemma}\label{solns-prod-linear}
    For \( \epsilon > 0 \), there exist positive constants \( C, \delta > 0 \) such that for \( X \) large enough, \( (\log X)^C \leq y \leq X \) and \( 1 \leq a, |b| \leq X^\delta \), we have
    \[
    \sum_{n \sim X} \I{n} \I{an + b} \ll_\epsilon X \rho(u)^{\phi - \epsilon},
    \]
    where \( u := \frac{\log X}{\log y} \) and \( \phi := \frac{13}{8} \).
\end{lemma}

\begin{proof}
    The result follows from \cite[Corollaire 4.2]{delabreteche-drappeau} upon noticing that one can replace the \( \frac{3}{5} \) in \cite[Théorème 2.1]{delabreteche-drappeau} by \( \frac{5}{8} \). See Remark (4) following \cite[Theorem 1.3]{pascadi}.
\end{proof}
\section{Defining the Dirichlet Polynomial}\label{Setting up notation}
Let \( X \geq 2 \) be large, and let \( y \) be the smoothness parameter such that
\begin{align}\label{y range}
    \exp\pare{C\frac{\log\log X}{\sigma_0}}=\exp\pare{\frac{C}{A_{vk}} (\log X)^{2/3} (\log \log X)^{4/3}} \leq y \leq X^{\frac{1}{C}},
\end{align}
where \( C \geq 1 \) is a large but fixed constant to be chosen later. Additionally, define
\begin{align}\label{u defn}
    u := \frac{\log X}{\log y}.
\end{align} 
Note that
\begin{align}\label{u bound}
    C \leq u \leq \frac{A_{vk} (\log X)^{1/3}}{C (\log \log X)^{4/3}} = \frac{\sigma_0\log X}{C\log \log X}.
\end{align}
From (\ref{y range}) and (\ref{u bound}), it follows that for arbitrarily large \( A \), we have
\begin{align}\label{y very large}
    \pare{\frac{\log X}{\rho(u)}}^A \ll_A y.
\end{align}

As discussed in the introduction, we will prove Theorem \ref{almost all interval} by showing that for a set of weights \( \{w_n\}_n \) supported on \( \textbf{S}(y) \), we have
\[
\sum_{x \leq n \leq x+h} w_n > 0 \quad \text{for almost all} \ x \in [X, 2X],
\]
where \( h \) is the length of the short interval. To define the weights, we introduce some notation.

Let
\begin{align}\label{J defn}
    J := \ceil{\frac{200u \log u}{\sigma_0 \log y}} = \ceil{\frac{200 (\log X)^{2/3} (\log \log X)^{1/3} u \log u}{A_{vk} \log y}} \leq \frac{200}{C} u + 1
\end{align}
and
\begin{align}\label{v defn}
    v := \frac{J \log \left( \frac{y}{2} \right)}{\log y}.
\end{align}
\begin{remark}
    To get a rough idea of the size of the quantities above, consider \( y = \exp\pare{ (\log X)^\tau } \) for some \( \tau \in (\frac{2}{3}, 1) \). Then, we have \( u = (\log X)^{1-\tau} \) and
\[
    v \approx J \asymp 1 + (\log X)^{5/3 - 2\tau}(\log\log X)^{4/3}.
\]
\end{remark}

We will always assume that \( C \) is large enough so that \( u > 3v \) holds. For \( \eta > 0 \) small enough to be chosen later, let \( P_1 \), \( P_2 \), and \( P_3 \) be parameters such that
\begin{align}\label{prime sizes}
    \frac{(\log X)^4}{\rho(u-v)^{2/3}} \leq P_1 \leq \frac{(\log X)^{20}}{\rho(u-v)^{14/3}}, \quad P_2 := \left( P_1 (\log X)^J \right)^{\frac{1}{\eta}}, \quad \text{and} \quad P_3 := \frac{y}{2}.
\end{align}
We will optimize the choice of \( P_1 \) later. It will turn out that, in the course of proving Theorem \ref{almost all interval} and Theorem \ref{all interval}, we will choose:

\[
P_1 \approx (\log X)^4 \rho(u-v)^{-3/4},
\]
for Theorem \ref{almost all interval}, and
\[
P_1 \approx (\log X)^4\rho(u-v)^{-3/2},
\]
for Theorem \ref{all interval}.

The parameter \( P_1 \) essentially determines the length of the short intervals in the statements of these theorems. For example, the short interval length in Theorem \ref{almost all interval} is approximately:
\[
\rho(u-v)^{-5/8} P_1 \approx (\log X)^4 \rho(u-v)^{-11/8}.
\]

Define
\begin{align}\label{M defn}
    \mathcal{M} := \textbf{S}(y) \cap \left[\frac{X}{2^{J+5} P_1 P_2 P_3^J}, \frac{8X}{P_1 P_2 P_3^J}\right].
\end{align}

With the above choice of parameters, we define our weights to be
\begin{align}\label{weights}
    w_n := \sum_{\substack{n = q_1 q_2 p_1 \dots p_J m \\ q_1 \sim P_1, q_2 \sim P_2 \\ P_3 < p_1, \dots, p_J \leq 2P_3 \\ m \in \mathcal{M}}} 1.
\end{align}
As a consequence of Lemma \ref{notational ease} (ii) below, \eqref{prime sizes} and \eqref{y very large}, we have that \[P_1\leq P_2\leq P_3\leq y/2.\] Hence the above weights are supported on \( \textbf{S}(y) \cap \left[\frac{X}{2^{J+5}}, 2^{J+5} X\right] \).

Now, let \( F(s) \) be the Dirichlet polynomial with the above weights as coefficients:
\[
F(s) := \sum_n \frac{w_n}{n^s}.
\]
It follows immediately from (\ref{weights}) that
\begin{align}\label{factorisation}
    F(s) = P_1(s) P_2(s) P_3(s)^J M(s),
\end{align}
where
\[
P_j(s) := \sum_{p \sim P_j} \frac{1}{p^s} \quad \text{for} \ j \in \{1, 2, 3\},
\]
and
\[
M(s) := \sum_{m \in \mathcal{M}} \frac{1}{m^s}.
\]
We will use this notation going forward without further mention.

The presence of multiple prime-supported polynomials can make the expressions notationally cumbersome. In the next lemma, we collect some simple observations that will help simplify the expressions during our analysis.

\begin{lemma}\label{notational ease}
    For the choice of parameters above, the following statements hold:
    \begin{enumerate}[(i)]
        \item For any \( \epsilon > 0 \), there exists a positive constant \( C_0(\epsilon) \) such that
        \[
        \rho(u) \leq \rho(u-v) \ll \rho(u)^{1-\epsilon},
        \]
        provided that \( C \geq C_0(\epsilon) \).
        
        \item Let \( A > 0 \) be a fixed positive constant. Then, for any \( \epsilon > 0 \), there exists a positive constant \( C_1(\epsilon) \) such that
        \[
        (A \log X)^J \ll_A \frac{\log X}{\rho(u-v)^\epsilon},
        \]
        provided that \( C \geq C_1(\epsilon) \).
        
        \item Let \( A> 0 \) and \(\kappa\) be fixed constants. Let $\tau$ be such that $|\tau|\leq 1$. Then, we have
        \[
        \rho\left(u-\tau v - \kappa\frac{\log(A^J P_1 P_2)}{\log y}\right) \asymp_{A,\kappa, \eta} \rho(u-\tau v).
        \]
    \end{enumerate}
\end{lemma}
\begin{proof}
    We begin by proving part (i). The first inequality follows from the fact that \( \rho \) is a decreasing function. For the other inequality, observe that since \( u > 3v \), by applying Lemma \ref{dickman-behaviour}, we get
    \[
    \rho(u-v) \leq \rho(u) \exp\left( 2J \log u \right).
    \]
    Now, substituting the upper bound for \( J \) from (\ref{J defn}), and for large enough \( C \) in terms of \( \epsilon \), we see that
    \[
    \rho(u-v) \leq u^2 \rho(u) \exp\left( \frac{400}{C} u \log u \right) \ll \rho(u)^{1-\epsilon}.
    \]

    For part (ii), observe that by substituting the value of \( J \) from (\ref{J defn}), we obtain
    \[
    (A \log X)^J = \exp\left( J (\log A + \log \log X) \right) \ll_A \log X \exp\left( \frac{200 u \log u (\log \log X + \log A)}{\sigma_0 \log y} \right).
    \]
    For large enough \( X \) in terms of \( A \), we have \( \log \log X + \log A \leq 2 \log \log X \), so
    \[
    (A \log X)^J \ll_A \log X \exp\left( \frac{400 u \log u \log \log X}{\sigma_0 \log y} \right).
    \]
    Using the lower bound for \( y \) in (\ref{y range}), we get the following bound:
    \[
    (A \log X)^J \ll_A \log X \exp\left( \frac{400 u \log u}{C} \right).
    \]
    Finally, by using part (i) and choosing \( C \) large enough in terms of \( \epsilon \), we obtain
    \[
    (A \log X)^J \ll_A \frac{\log X}{\rho(u)^{\epsilon/2}} \ll \frac{\log X}{\rho(u-v)^\epsilon}.
    \]

    By part (ii), \eqref{prime sizes} and \eqref{y very large}, we have
\[
    A^J P_1 P_2 \ll_A\pare{\frac{\rho(u-v)}{\log X}}^{20+\frac{100}{\eta}}\leq (y/2)^{J}.
\]
Hence, recalling the definition of $v$ and the fact that \( u > 3v \), it follows that
\[
    \frac{\log(A^J P_1 P_2)}{\log y} \leq v \leq \frac{u-v}{2}.
\]
    
    Thus, by Lemma \ref{dickman-behaviour}, we have
    \[
    \rho\left( u-\tau v - \kappa\frac{\log(A^J P_1 P_2)}{\log y} \right) = \rho(u-\tau v) \exp\left( (1 + o(1))\kappa \frac{\log(A^J P_1 P_2)}{\log y} \log(u-\tau v) \right).
    \]
    To bound the right-hand side, we note that
    \[
    \frac{\log(A^J P_1 P_2)}{\log y} \log(u-\tau v) \ll_A \frac{\log P_2 \log u}{\log y} \ll_{A, \eta} \frac{\log \log X \log u + u (\log u)^2}{\log y}.
    \]
    Using (\ref{u bound}), we see that
    \[
    \frac{\log \log X \log u + u (\log u)^2}{\log y} = \frac{u \log u \log \log X}{\log X} + \frac{u^2 (\log u)^2}{\log X} = O(1).
    \]
    Therefore, we conclude that
    \[
    \rho\left( u-\tau v - \kappa\frac{\log(A^J P_1 P_2)}{\log y} \right) = \rho(u-\tau v) \exp(O_{A,\kappa, \eta}(1)) \asymp_{A,\kappa, \eta} \rho(u-\tau v).
    \]
\end{proof}
Part (i) of the above lemma allows us to interchange \( \rho(u-v) \) and \( \rho(u) \) freely. Part (ii) essentially states that, by sacrificing a factor of \( \log X \), we can ignore the \( (\log X)^J \) term that arises from multiplicity in the counting, due to the presence of multiple prime-supported polynomials. The loss of the \( \log X \) factor is only significant for larger values of \( y \). Finally, part (iii) enables us to express the number of \( y \)-smooth numbers of size \( \asymp \frac{X}{(AP_3)^J P_1 P_2} \) in a convenient form.

From (\ref{weights}) and Lemma \ref{notational ease} (ii), for large enough \( C \) in terms of \( \epsilon \), we conclude that 
\begin{align}\label{wt bound}
    w_n \leq \frac{(\log n)^{J+2}}{\log P_1 \log P_2 (\log P_3)^J} 
    \ll \frac{(2 \log X)^{J+2}}{\log P_1 \log P_2 (\log P_3)^J} 
    \ll \frac{(\log X)^3}{\rho(u-v)^\epsilon}.
\end{align}
In the following lemma, we derive a lower bound for the averages of our weights over moderately short intervals.
\begin{lemma}\label{weights av}
    Let \( \epsilon > 0 \). There exists a positive constant $C(\epsilon)$ such that for \( x \in [X, 2X] \) and \( 2X y^{-5/12} \leq h \leq X \), 
    \begin{align*}
        \frac{1}{h} \sum_{x \leq n \leq x + h} w_n \gg_{\eta} \frac{\rho(u-v)}{\log P_1 \log P_2 (2 \log P_3)^J} \gg_\eta \left( \frac{\rho(u-v)}{\log X} \right)^{1+\epsilon}
    \end{align*}
    holds uniformly in \( x \), provided \( C \geq C(\epsilon) \).
\end{lemma}
\begin{proof}
    Using (\ref{weights}), we have that 
    \begin{align}\label{wts av}
        \frac{1}{h}\subsum{x\leq n\leq x+h}w_n = \frac{1}{h}\sum_{q_1\sim P_1}\sum_{q_2\sim P_2}\sum_{P_3<p_1, \dots, p_J\leq 2P_3}\subsum{\frac{x}{q_1q_2p_1\dots p_J}\leq m\leq \frac{x+h}{q_1q_2p_1\dots p_J}\\m\in\textbf{S}(y)}1.
    \end{align}
    The innermost summation is over smooth numbers in a short interval. From our hypothesis, it follows that 
    \[
    \frac{x}{q_1 q_2 p_1 \dots p_J} y^{-5/12} \leq \frac{h}{q_1 q_2 p_1 \dots p_J} \leq \frac{x}{q_1 q_2 p_1 \dots p_J}.
    \]
    Therefore, we can apply Hildebrand's short interval estimate from Lemma \ref{smooth-number-estimates} (ii). This gives
        \begin{align}\label{inner sum smooth}
       \subsum{\frac{x}{q_1q_2p_1\dots p_J}\leq m\leq \frac{x+h}{q_1q_2p_1\dots p_J}\\m\in \textbf{S}(y)}1\gg \frac{h}{q_1q_2p_1\dots p_J}\rho\pare{\frac{\log\frac{x}{q_1q_2p_1\dots p_J}}{\log y}}.
    \end{align}
    Combining (\ref{wts av}) and (\ref{inner sum smooth}), we obtain:
    \begin{align*}
        \frac{1}{h} \sum_{x \leq n \leq x + h} w_n & \gg \sum_{q_1 \sim P_1} \sum_{q_2 \sim P_2} \sum_{P_3 < p_1, \dots, p_J \leq 2P_3} \frac{1}{q_1 q_2 p_1 \dots p_J} \rho \left( \frac{\log \frac{x}{q_1 q_2 p_1 \dots p_J}}{\log y} \right) \\
        & \gg \rho \left( \frac{\log \frac{X}{P_1 P_2 P_3^J}}{\log y} \right) \frac{1}{\log P_1 \log P_2 (2 \log P_3)^J}.
    \end{align*}

    Finally, applying Lemma \ref{notational ease} (iii) and (ii), and using (\ref{prime sizes}), we conclude that:
    \[
    \frac{1}{h} \sum_{x \leq n \leq x + h} w_n \gg_{\eta} \frac{\rho(u-v)}{\log P_1 \log P_2 (2 \log P_3)^J} \geq \frac{\rho(u-v)}{\log P_1 \log P_2 (2 \log X)^J}.
    \]
    From \eqref{prime sizes} and Lemma \ref{notational ease} (ii), we have 
    \[
    P_2 \ll_\epsilon \pare{\frac{P_1\log X}{\rho(u-v)^\epsilon}}^{\frac{1}{\eta}}.
    \]
Substituting the above bound for \( P_2 \) and the upper bound for \( P_1 \) from \eqref{prime sizes}, 
together with an application of Lemma \ref{notational ease}, we arrive at the inequality:
\begin{align}\label{second lower bound}
    \frac{\rho(u-v)}{\log P_1 \log P_2 (2 \log P_3)^J} 
\gg_{\eta} \left( \frac{\rho(u-v)}{\log X} \right)^{1 + \epsilon}.
\end{align}
\end{proof}
\section{Parseval Reduction}\label{parseval reduction}
In this section, we recall a Parseval-type bound that will help reduce the problem of counting smooth numbers in almost all short intervals to establishing a non-trivial upper bound for the mean value of the Dirichlet polynomial \(F(s)\).
\begin{lemma}\label{parseval}  
    Suppose \( X, T_0 > 0 \) are such that \( X \geq 2T_0^3 \), and let \( J \) be a positive integer. Let \( \{a_n\}_n \) be a sequence of complex numbers supported on \( \mathcal{I} = \left[\frac{X}{2^{J+5}}, 2^{J+5}X\right] \). For \( x \in [X, 2X] \) and \( 2 \leq h_1 \leq h_2 \leq \frac{X}{T_0^3} \), define
    \[
    S_{h_j}(x) := \frac{1}{h_j} \sum_{x \leq n \leq x + h_j} a_n \quad \text{for} \quad j = 1, 2.
    \]
    Finally, let 
    \[
    G(s) := \sum_n \frac{a_n}{n^s}.
    \]
    Then, we have the following bound:
    \[
    \frac{1}{X} \int_X^{2X} \left| \frac{1}{h_1} S_{h_1}(x) - \frac{1}{h_2} S_{h_2}(x) \right|^2 \, dx \ll \frac{(J \max_{n \in \mathcal{I}} |a_n|)^2}{T_0} + \int_{T_0}^{\frac{X}{h_1}} |G(1+it)|^2 \, dt 
    \]
    \[
    \hspace{3in}+ \max_{T \geq \frac{X}{h_1}} \frac{X}{T h_1} \int_T^{2T} |G(1+it)|^2 \, dt.
    \]
\end{lemma}    
\begin{proof}
    Notice that
    \begin{align}\label{polynomial has different bound}
        \abs{G(1+it)} = \abs{\sum_{n \in \mathcal{I}} \frac{a_n}{n^{1+it}}} \ll \max_{n \in \mathcal{I}} |a_n| \sum_{n \in \mathcal{I}} \frac{1}{n} \ll J \max_{n \in \mathcal{I}} |a_n|.
    \end{align}
    The proof is essentially the same as in \cite[Lemma 14]{matomaki-radziwill}, with two key differences. First, we do not specify \( T_0 \), and second, instead of the bound \( |G(1+it)| = O(1) \), we use the estimate in (\ref{polynomial has different bound}). See also \cite[Lemma 1]{teravainen}.
\end{proof}
Let \( S_{h_j}(x) := \sum_{x \leq n \leq x + h_j} w_n \) for \( j \in \{1, 2\} \), where \( h_1 \geq 2 \) will be chosen later, and \( h_2 := X y^{-3/8} \). By Lemma \ref{parseval}, for \( T_0 = y^{1/8} \), we have
\begin{align}\label{par app}
    \begin{aligned}
    \frac{1}{X} \int_{X}^{2X} \left| \frac{1}{h_1} S_{h_1}(x) - \frac{1}{h_2} S_{h_2}(x) \right|^2 dx \ll & \frac{(J \max_{n} w_n)^2}{y^{1/8}} + \int_{y^{1/8}}^{X / h_1} |F(1+it)|^2 dt \\
    & + \max_{T \geq \frac{X}{h_1}} \frac{X}{T h_1} \int_{T}^{2T} |F(1+it)|^2 dt.
    \end{aligned}
\end{align}

To apply the Chebyshev-style argument, we need to bound the right-hand side of the above inequality. The main challenge lies in bounding the mean values of \( F(s) \). In the next section, we will establish a non-trivial upper bound for the mean value of \( F(s) \), and in the section following that, we will apply this result to complete the proof of Theorem \ref{almost all interval}.
\section{Bounding the Mean Value}\label{main content}
In this section, we will show that \(F(s)\) satisfies the following mean value estimate by exploiting the factorisation in (\ref{factorisation}).
\begin{proposition}\label{main propn}
    Let \( \epsilon > 0 \) be small, and let \( \phi := \frac{13}{8} \). Let \( F(s) \) be as defined in (\ref{factorisation}). Then, there exist positive constants \( X_0(\eta) \) and \( C(\epsilon, \eta) \) such that for all \( X \geq X_0(\eta) \), any fixed \( C \geq C(\epsilon, \eta) \), and any \( T \geq 1 \), the following bound holds:
    \begin{align*}
        \int_{y^{1/8}}^{T} \left| F(1+it) \right|^2 \, dt 
        \ll_{\epsilon,\eta} & \, P_1^{-\frac{1}{2} + 64\eta} \left( \frac{P_1 T}{X} \rho(u-v)^{1-\epsilon} + \rho(u-v)^{\phi - \epsilon} \right) \frac{(\log X)^2}{(\log P_2 (\log P_3)^J)^2} \\
        & \hspace{3in}+ \left( \frac{\rho(u-v)}{(\log X)^J} \right)^{25}.
    \end{align*}
\end{proposition}
Before proving the proposition, we will establish several intermediate lemmas. Throughout this section, \( C_1, C_2, \dots \) will denote large positive constants that are independent of the parameters defined in Section \ref{Setting up notation}, specifically \( X, y, P_1, P_2, P_3, \epsilon, \) and \( \eta \). Additionally, \( \sum{\vphantom{\sum}}^\ddag \) indicates that the summation variable runs over powers of \( 2 \).

The first lemma will assist us in obtaining an upper bound for the contribution from the range of integration where the Dirichlet polynomial \( P_1(s) \) is small in magnitude.
\begin{lemma}\label{contri T1}
    Let \( \epsilon > 0 \) be small, and define \( \phi := \frac{13}{8} \). Let \( A(s) := P_2(s) P_3(s)^J M(s) \). Then, there exists a positive constant \( C(\epsilon) \) such that for any fixed \( C \geq C(\epsilon) \) and any \( T \geq 1 \), we have the following bound:
    \[
        \int_{-T}^{T} |A(1+it)|^2 \, dt \ll_{\epsilon, \eta} \left( \frac{P_1 T}{X} \rho(u-v)^{1-\epsilon} + \rho(u-v)^{\phi-\epsilon} \right) \frac{(\log X)^2}{(\log P_2 (\log P_3)^J)^2}.
    \]
\end{lemma}
\begin{proof}
    From the definition of \( A(s) \), it is easy to see that it can be written as \( \sum_{n} \frac{a_n}{n^s} \) for
    \begin{align}\label{coeff bound T1}
        a_n := \sum_{\substack{n = q_2 p_1 \dots p_J m \\ q_2 \sim P_2 \\ P_3 < p_1, \dots, p_J \leq 2P_3 \\ m \in \mathcal{M}}} 1 \leq \frac{(\log n)^{J+1}}{\log P_2 (\log P_3)^J}.
    \end{align}
    Note that \( \{a_n\}_n \) is supported on \( n \in \left[\frac{X}{2^{J+5} P_1}, \frac{2^{J+4} X}{P_1}\right] \). Let \( \mathcal{I} := \left[\frac{X}{2^{J+6} P_1}, \frac{2^{J+5} X}{P_1}\right] \).
    
    We perform dyadic division and then apply the Cauchy-Schwarz inequality to obtain:
    \begin{align*}
        \int_{-T}^{T} |A(1+it)|^2 \, dt &\ll \int_{-T}^T \left| \sum_{M\in \mathcal{I}}{\vphantom{\sum}}^\ddag \sum_{n \sim M} \frac{a_n}{n^{1+it}} \right|^2 dt \\
        &\ll J \sum_{M\in \mathcal{I}}{\vphantom{\sum}}^\ddag \int_{-T}^T \left| \sum_{n \sim M} \frac{a_n}{n^{1+it}} \right|^2 dt.
    \end{align*}
    Now, we are in a position to apply Lemma \ref{improved mvt} to the integral. By doing so, we get:
    \begin{align}\label{improved MVT output}
        \int_{-T}^{T} |A(1+it)|^2 \, dt &\ll JT \sum_{M\in \mathcal{I}}{\vphantom{\sum}}^\ddag \frac{1}{M^2} \left( \sum_{n \sim M} |a_n|^2 + \sum_{k \leq \frac{2M}{T}} \sum_{n \sim M} |a_n a_{n+k}| \right) \nonumber \\
        &\ll JT \sum_{M\in \mathcal{I}}{\vphantom{\sum}}^\ddag \frac{1}{M^2} \left( \sum_{n \sim M} |a_n|^2 + \sum_{k \leq \frac{2M}{T}} \sum_{\substack{n \sim M \\ n+k \in \mathbf{S}(y)}} |a_n| \right) \max_{M \leq n \leq 4M} |a_n|.
    \end{align}
        To bound the diagonal contribution (i.e., the first term of the outer summation above), we observe that by using (\ref{coeff bound T1}), Lemma \ref{smooth-number-estimates} (ii), and Lemma \ref{notational ease} (iii), we get
    \begin{align}\label{no cross smooth}
        \sum_{n \sim M} |a_n| &\ll \sum_{q_2 \sim P_2} \sum_{P_3 < p_1, \dots, p_J \leq 2P_3} \sum_{\substack{m \in \mathcal{M} \\ m \sim \frac{M}{q_2 p_1 \dots p_J}}} 1 \nonumber \\
        &\ll_{\eta} C_1^J \rho(u-v) \frac{M}{\log P_2 (\log P_3)^J}.
    \end{align}
    
    For the off-diagonal contribution (i.e., the remaining term), we first note that by (\ref{coeff bound T1}) and the fact that \( \mathcal{M} \subseteq \mathbf{S}(y) \), it follows that
    \begin{align*}
        \sum_{\substack{n \sim M \\ n+k \in \mathbf{S}(y)}} |a_n| &\ll \sum_{q_2 \sim P_2} \sum_{P_3 < p_1, \dots, p_J \leq 2P_3} \sum_{\substack{m \sim \frac{M}{q_2 p_1 \dots p_J} \\ m (q_2 p_1 \dots p_J m + k) \in \mathbf{S}(y)}} 1.
    \end{align*}
    From (\ref{prime sizes}) and (\ref{J defn}), it follows that \( P_2 (2P_3)^J \leq y^{J+1} \leq y^2 \exp\left(200 \frac{u \log u}{\sigma_0}\right) \leq X^{\frac{202}{C}} \). Hence, for large enough \( C \) in terms of \( \epsilon \), we may apply Lemma \ref{solns-prod-linear} to the innermost summation and simplify using Lemma \ref{notational ease} (iii) to obtain
    \begin{align}\label{cross smooth}
        \sum_{\substack{n \sim M \\ n+k \in \mathbf{S}(y)}} |a_n| &\ll_{\epsilon, \eta} C_2^J \rho(u-v)^{\phi - \frac{\epsilon}{4}} \frac{M}{\log P_2 (\log P_3)^J}.
    \end{align}
    
    Substituting the bounds in (\ref{coeff bound T1}), (\ref{no cross smooth}), and (\ref{cross smooth}) into (\ref{improved MVT output}), we obtain
    \begin{align*}
        \int_{-T}^{T} |A(1+it)|^2 \, dt &\ll_{\epsilon, \eta} \left( \frac{P_1 T}{X} \rho(u-v) + \rho(u-v)^{\phi - \frac{\epsilon}{4}} \right) \frac{(C_3 \log X)^{J+1}}{(\log P_2 (\log P_3)^J)^2}.
    \end{align*}
    
    Finally, applying Lemma \ref{notational ease} (ii) by choosing \( C \) even larger in terms of \( \epsilon \) if needed, we obtain the desired bound:
    \begin{align*}
        \int_{-T}^{T} |A(1+it)|^2 \, dt &\ll_{\epsilon, \eta} \left( \frac{P_1 T}{X} \rho(u-v)^{1 - \frac{\epsilon}{4}} + \rho(u-v)^{\phi - \frac{\epsilon}{2}} \right) \frac{(\log X)^2}{(\log P_2 (\log P_3)^J)^2} \\
        &\ll_{\epsilon, \eta} \left( \frac{P_1 T}{X} \rho(u-v)^{1-\epsilon} + \rho(u-v)^{\phi - \epsilon} \right) \frac{(\log X)^2}{(\log P_2 (\log P_3)^J)^2}.
    \end{align*}
\end{proof}
The next lemma provides an upper bound for the contribution from the region of integration where \( P_2(s) \) is small, but \( P_1(s) \) is not too small. The key idea is to amplify the length of the Dirichlet polynomial by introducing a moment of \( P_1(s) \).
\begin{lemma}\label{contri T2}
    Let \( A(s) := P_1(s) P_3(s)^J M(s) \). Additionally, for \( \alpha > 0 \) and \( 1 \leq T \leq X \), let \( \mathcal{T} \subseteq [-T, T] \) be a set such that \( |P_1(1+it)| \geq P_1^{-\alpha} \) for all \( t \in \mathcal{T} \). Then, for \( \ell := \ceil{\frac{\log P_2}{\log P_1}} \), we have the following bound:
    \[
        \int_{\mathcal{T}} |A(1+it)|^2 \, dt \ll P_1^{2\alpha \ell} \pare{\frac{T}{X} + P_1} \pare{\frac{(C_4 \log X)^{J+1} \log P_2}{\log P_1 (\log P_3)^J}}^2 \times \exp \pare{ 2 \frac{\log P_2}{\log P_1} \log \log P_2 }.
    \]
\end{lemma}
\begin{proof}
    We begin by performing similar manipulations as in the proof of the previous lemma. We can write \( A(s) = \sum_{n} \frac{b_n}{n^s} \) for
    \begin{align}\label{coeff bound T2}
        b_n := \sum_{\substack{n = q_1 p_1 \dots p_J m \\ q_1 \sim P_1 \\ P_3 < p_1, \dots, p_J \leq 2P_3 \\ m \in \mathcal{M}}} 1 \leq \frac{(\log n)^{J+1}}{\log P_1 (\log P_3)^J}.
    \end{align}
    Note that \( \{b_n\}_n \) is supported on \( n \in \left[\frac{X}{2^{J+5} P_2}, \frac{2^{J+4} X}{P_2}\right] \). Let \( \mathcal{I} := \left[\frac{X}{2^{J+6}}, 2^{J+5} X\right] \).
    
    Using dyadic division and the Cauchy-Schwarz inequality, we obtain:
    \[
        \int_{\mathcal{T}} |A(1+it)|^2 \, dt = \int_{\mathcal{T}} \left| \sum_{M \in \mathcal{I}} {\vphantom{\sum}}^\ddag \sum_{n \sim \frac{M}{P_2}} \frac{b_n}{n^{1+it}} \right|^2 dt 
        \ll J \sum_{M \in \mathcal{I}} {\vphantom{\sum}}^\ddag \int_{\mathcal{T}} \left| \sum_{n \sim \frac{M}{P_2}} \frac{b_n}{n^{1+it}} \right|^2 dt.
    \]
        Now we use the lower bound on \( P_1(s) \) to increase the length of the Dirichlet polynomial inside the integral. Observe that from \( |P_1(1+it)| > P_1^{-\alpha} \), it follows that 
    \[
        |P_1(1+it) P_1^\alpha|^{2\ell} \geq 1 \quad \text{for all} \quad \ell \in \mathbb{N}.
    \]
    In particular, for \( \ell = \left\lceil \frac{\log P_2}{\log P_1} \right\rceil \), we have
    \begin{align*}
        \int_{\mathcal{T}} |A(1+it)|^2 \, dt &\ll J P_1^{2\alpha \ell} \sum_{M \in \mathcal{I}} {\vphantom{\sum}}^\ddag \int_{\mathcal{T}} \left| P_1(1+it)^\ell \sum_{n \sim \frac{M}{P_2}} \frac{b_n}{n^{1+it}} \right|^2 dt \\
        &\leq J P_1^{2\alpha \ell} \sum_{M \in \mathcal{I}} {\vphantom{\sum}}^\ddag \int_{-T}^{T} \left| P_1(1+it)^\ell \sum_{n \sim \frac{M}{P_2}} \frac{b_n}{n^{1+it}} \right|^2 dt.
    \end{align*}
    
    Next, applying Lemma \ref{moment computation} followed by the bound in (\ref{coeff bound T2}), we get:
    \begin{align}\label{without stirling T2}
        \int_{\mathcal{T}} |A(1+it)|^2 \, dt &\ll J P_1^{2\alpha \ell} \sum_{M \in \mathcal{I}} {\vphantom{\sum}}^\ddag \left( \frac{T}{M} + 2^\ell P_1 \right) (\ell+1)!^2 \max_{n \sim M} |b_n|^2 \nonumber\\
        &\ll P_1^{2\alpha \ell} \left( \frac{T}{X} + 2^\ell P_1 \right) (\ell+1)!^2 \left( \frac{\left( C_4 \log X \right)^{J+1}}{\log P_1 (\log P_3)^J} \right)^2.
    \end{align}
    To simplify the above expression, we use Stirling's approximation, which gives \( 2^\ell (\ell+1)!^2 \ll \exp(2\ell \log \ell) \). Furthermore, by the mean value theorem
    \[
        (x+1)\log(x+1) \leq x \log x + \log x + 2 \quad \text{for large enough } x.
    \]
    So it follows that
    \[
        \ell \log \ell \leq \frac{\log P_2}{\log P_1} \log \frac{\log P_2}{\log P_1} + \log \frac{\log P_2}{\log P_1} + 2.
    \]
    Thus, we obtain the bound
    \[
        2^\ell (\ell+1)! \ll (\log P_2)^2 \exp \left( 2 \frac{\log P_2}{\log P_1} \log \log P_2 \right).
    \]
    Substituting this bound into (\ref{without stirling T2}), we get the desired estimate:
    \[
            \int_{\mathcal{T}} |A(1+it)|^2 \, dt \ll P_1^{2\alpha \ell} \pare{\frac{T}{X} + P_1} \pare{\frac{(C_4 \log X)^{J+1} \log P_2}{\log P_1 (\log P_3)^J}}^2 \times \exp \pare{ 2 \frac{\log P_2}{\log P_1} \log \log P_2 }.
        \]
\end{proof}
Finally, we will need the following lemma to upper bound the contribution from the remaining range of integration. In this range, we bound the polynomial \( P_3(s) \) using the pointwise bound provided by Lemma \ref{prime-poly-pw-bound}.
\begin{lemma}\label{contri T3}
    Let \( A(s) := P_1(s)P_2(s)M(s) \). For \( T \geq 1 \), let \( \mathcal{T} \subseteq [-T, T] \) be a set. Then there exists a well-spaced set \( \mathcal{U} \subseteq \mathcal{T} \) such that
    \[
        \int_{\mathcal{T}} |A(1+it)|^2 \, dt \ll_{\eta} J^2 \log T \left( 1 + |\mathcal{U}| \sqrt{T} \frac{(2P_3)^J}{X} \right) \left( \frac{\log X}{\log P_1 \log P_2} \right)^2 \rho(u-v).
    \]
\end{lemma}
\begin{proof}
    We can write \( A(s) = \sum_{n} \frac{c_n}{n^s} \) for
    \begin{align}\label{coeff bound T3}
        c_n := \sum_{\substack{n = q_1 q_2 m \\ q_j \sim P_j \\ m \in \mathcal{M}}} 1 \leq \frac{(\log n)^2}{\log P_1 \log P_2}.
    \end{align}
    Note that the sequence \( \{c_n\}_n \) is supported on \( n \in \left[ \frac{X}{2^{J+5} P_3^J}, \frac{32 X}{P_3^J} \right] \).

    We can now find a well-spaced set \( \mathcal{U} \subseteq \mathcal{T} \) such that
    \[
        \int_{\mathcal{T}} |A(1+it)|^2 \, dt \ll \sum_{t \in \mathcal{U}} |A(1+it)|^2.
    \]
    Let \( \mathcal{I} := \left[ \frac{X}{2^{J+6} P_3^J}, \frac{64 X}{P_3^J} \right] \). Using dyadic division and the Cauchy-Schwarz inequality, we have
    \begin{align*}
        \int_{\mathcal{T}} |A(1+it)|^2 \, dt & \ll \sum_{t \in \mathcal{U}} \left| \sum_{M \in \mathcal{I}} {\vphantom{\sum}}^\ddag \sum_{n \sim M} \frac{c_n}{n^{1+it}} \right|^2 \\
        & \ll J \sum_{M \in \mathcal{I}} {\vphantom{\sum}}^\ddag \sum_{t \in \mathcal{U}} \left| \sum_{n \sim M} \frac{c_n}{n^{1+it}} \right|^2.
    \end{align*}
    Applying Lemma \ref{halasz-montgomery}, we get
    \begin{align*}
        \int_{\mathcal{T}} |A(1+it)|^2 \, dt & \ll J \log T \sum_{M \in \mathcal{I}} {\vphantom{\sum}}^\ddag \left( M + |\mathcal{U}| \sqrt{T} \right) \sum_{n \sim M} \left| \frac{c_n}{n} \right|^2 \\
        & \ll J \log T \sum_{M \in \mathcal{I}} {\vphantom{\sum}}^\ddag \left( M + |\mathcal{U}| \sqrt{T} \right) \frac{\max_{n \sim M} c_n}{M^2} \sum_{n \sim M} c_n.
    \end{align*}
    Additionally, using Lemma \ref{smooth-number-estimates} (i) and Lemma \ref{notational ease} (iii), we obtain
    \[
        \sum_{n \sim M} c_n = \sum_{q_1 \sim P_1} \sum_{q_2 \sim P_2} \sum_{\substack{m \in \mathcal{M} \\ m \sim \frac{M}{q_1 q_2}}} 1 
        \ll \sum_{q_1 \sim P_1} \sum_{q_2 \sim P_2} \frac{M}{q_1 q_2} \rho \left( \frac{\log \frac{M}{q_1 q_2}}{\log y} \right)
        \ll_{\eta} \frac{M}{\log P_1 \log P_2} \rho(u-v).
    \]
    Finally, combining these estimates with (\ref{coeff bound T3}), we get
    \[
        \int_{\mathcal{T}} |A(1+it)|^2 \, dt \ll_{\eta} J^2 \log T \left( 1 + |\mathcal{U}| \sqrt{T} \frac{(2 P_3)^J}{X} \right) \left( \frac{\log X}{\log P_1 \log P_2} \right)^2 \rho(u-v).
    \]
\end{proof}
Now we turn to the proof of Proposition \ref{main propn}.
\begin{proof}[Proof of Proposition \ref{main propn}]
        By Lemma \ref{MVT}, we have the bound
        \begin{align*}
            \int_{-T}^{T} |F(1+it)|^2 \, dt \ll (T + X) \sum_{n} \left( \frac{w_n}{n} \right)^2.
        \end{align*}
        Write \( \mathcal{I} := \left[ \frac{X}{2^{J+6}}, 2^{J+5} X \right] \). From (\ref{weights}) and (\ref{wt bound}), we obtain
        \begin{align*}
            \sum_{n} \left( \frac{w_n}{n} \right)^2 \ll \sum_{M \in \mathcal{I}} {\vphantom{\sum}}^\ddag \frac{1}{M^2} \left( \sum_{q_1 \sim P_1} \sum_{q_2 \sim P_2} \sum_{P_3 < p_1, \dots, p_J \leq 2P_3} \sum_{\substack{m \in \mathcal{M} \\ m \sim \frac{M}{q_1 q_2 p_1 \dots p_J}}} 1 \right) \frac{(2 \log X)^{J+2}}{\log P_1 \log P_2 (\log P_3)^J}.
        \end{align*}
        Using Lemma \ref{smooth-number-estimates} (ii), Lemma \ref{notational ease} (ii) and (iii), we conclude that
        \begin{align*}
            \sum_{n} \left( \frac{w_n}{n} \right)^2 & \ll_{\eta} \frac{\rho(u-v)^{1-\epsilon} (\log X)^3}{X (\log P_1 \log P_2 (\log P_3)^J)^2}.
        \end{align*}
    
        Hence, by combining the results, we get
        \begin{align*}
            \int_{-T}^{T} |F(1+it)|^2 \, dt \ll_{\eta} \left( \frac{T}{X} + 1 \right) \frac{\rho(u-v)^{1-\epsilon} (\log X)^3}{(\log P_1 \log P_2 (\log P_3)^J)^2}.
        \end{align*}
        Using the lower bound on \( P_1 \) from (\ref{prime sizes}), we observe that this bound is sufficient when \( T > X \). Therefore, we may assume that \( T \leq X \) going forward.
    
        Set \( \alpha_1 = \frac{1}{4} - 4\beta \), \( \alpha_2 = \frac{1}{4} - 2\beta \) for \( \beta = 8\eta \), and write the interval \( [y^{1/8}, T] \) as a union of three sets:
        \begin{align*}
            \mathcal{T}_1 &= \{ t \in [y^{1/8}, T] : |P_1(1+it)| \leq P_1^{-\alpha_1} \}, \\
            \mathcal{T}_2 &= \{ t \in [y^{1/8}, T] : |P_2(1+it)| \leq P_2^{-\alpha_2} \} \setminus \mathcal{T}_1, \\
            \mathcal{T}_3 &= [y^{1/8}, T] \setminus (\mathcal{T}_1 \cup \mathcal{T}_2).
        \end{align*}
        Thus, we can write the integral as
        \[
            \int_{y^{1/8}}^{T} = \int_{\mathcal{T}_1} + \int_{\mathcal{T}_2} + \int_{\mathcal{T}_3}.
        \]
    We will now show that the contribution from each of the $\mathcal{T}_j$ is bounded by the desired quantity.
    \subsection{Contribution from \texorpdfstring{\( \mathcal{T}_1 \)}{T1}}

    Let \( A_1(s) = P_2(s) P_3(s)^J M(s) \). As \( |P_1(1+it)| \leq P_1^{-\alpha_1} \) for \( t \in \mathcal{T}_1 \), we have
    \[
        \int_{\mathcal{T}_1} |F(1+it)|^2 \, dt \ll P_1^{-2\alpha_1} \int_{-T}^{T} |A_1(1+it)|^2 \, dt.
    \]
    Applying Lemma \ref{contri T1} to the integral on the right-hand side, we obtain
    \[
        \int_{\mathcal{T}_1} |F(1+it)|^2 \, dt \ll_{\epsilon, \eta} P_1^{-2\alpha_1} \left( \frac{P_1 T}{X} \rho(u-v)^{1-\epsilon} + \rho(u-v)^{\phi-\epsilon} \right) \frac{(\log X)^2}{(\log P_2 (\log P_3)^J)^2}.
    \]
        \subsection{Contribution from \texorpdfstring{\( \mathcal{T}_2 \)}{T2}}
    
    Let \( A_2(s) := P_1(s) P_3(s)^J M(s) \). As \( |P_2(1+it)| \leq P_2^{-\alpha_2} \) for \( t \in \mathcal{T}_2 \), we have
    \[
        \int_{\mathcal{T}_2} |F(1+it)|^2 \, dt \ll P_2^{-2\alpha_2} \int_{\mathcal{T}_2} |A_2(1+it)|^2 \, dt.
    \]
    By Lemma \ref{contri T2}, for \( \ell = \ceil{ \frac{\log P_2}{\log P_1} } \), we have
    \begin{align*}
        \int_{\mathcal{T}_2}|F(1+it)|^2dt\ll P_2^{-2\alpha_2}P_1^{2\alpha_1\ell}\pare{\frac{T}{X}+P_1}\pare{\frac{\pare{C_4\log X}^{J+1}\log P_2}{\log P_1(\log P_3)^{J}}}^2\times\\ \exp\pare{2\frac{\log P_2}{\log P_1}\log \log P_2}.
    \end{align*}
    Simplifying this expression, we obtain
    \begin{align*}
        \int_{\mathcal{T}_2}|F(1+it)|^2dt \ll \pare{\frac{T}{X}+1}\pare{\frac{\pare{C_5\log X}^{J+1}P_1\log P_2}{\log P_1(\log P_3)^{J}}}^2\times\\\exp\left(2\pare{\alpha_1-\alpha_2+\frac{\log \log P_2}{\log P_1}}\log P_2\right).
    \end{align*}
    Since \( \alpha_2 - \alpha_1 = 2\beta \) and \( \frac{\log \log P_2}{\log P_1} \leq \beta \) for large enough \( X \) depending on \( \eta \), we can simplify further to obtain
    \[
        \int_{\mathcal{T}_2} |F(1+it)|^2 \, dt \ll \left( \frac{T}{X} + 1 \right) \left( \frac{(C_5 \log X)^{J+1} P_1 \log P_2}{\log P_1 (\log P_3)^J} \right)^2 P_2^{-\beta}.
    \]
    Substituting \( P_2 = \left( P_1 (\log X)^J \right)^{\frac{1}{\eta}} \) and using \( P_1 \geq \rho(u-v)^{-2/3} \) from (\ref{prime sizes}), for small enough \( \eta \), we conclude that
    \[
        \int_{\mathcal{T}_2} |F(1+it)|^2 \, dt \ll_{\epsilon, \eta} P_1^{-2\alpha_1} \left( \frac{P_1 T}{X} \rho(u-v)^{1-\epsilon} + \rho(u-v)^{\phi-\epsilon} \right) \frac{(\log X)^2}{(\log P_2 (\log P_3)^J)^2}.
    \]
    \subsection{Contribution from \texorpdfstring{\( \mathcal{T}_3 \)}{T3}}
    
    Let \( A_3(s) = P_1(s) P_2(s) M(s) \). By Lemma \ref{prime-poly-pw-bound}, \( |P_3(1+it)| \ll P_3^{-\sigma_0} \) for \( t \geq y^{1/8} \). Hence, we have
    \[
        \int_{\mathcal{T}_3} |F(1+it)|^2 \, dt \ll P_3^{-2J \sigma_0} \int_{\mathcal{T}_3} |A_3(1+it)|^2 \, dt.
    \]
    Using Lemma \ref{contri T3}, we can find a well-spaced set \( \mathcal{T} \subseteq \mathcal{T}_3 \) such that
    \[
        \int_{\mathcal{T}_3} |F(1+it)|^2 \, dt \ll_{\eta} J^2 P_3^{-2J \sigma_0} \log T \left( 1 + |\mathcal{T}| \sqrt{T} \frac{(2 P_3)^J}{X} \right) \left( \frac{\log X}{\log P_1 \log P_2} \right)^2 \rho(u-v).
    \]
    Note that by the definition of \( P_2 \) and the fact that \( P_1 \geq \log X \) from (\ref{prime sizes}), it follows that
    \[
        \frac{\log \log T}{\log P_2} \leq \frac{\eta \log \log T}{\log P_1} \leq \eta.
    \]
    Since \( |P_2(1+it)| > P_2^{-\alpha_2} \) and \( \frac{\log \log T}{\log P_2} \leq \eta \), by Lemma \ref{prime-poly-large-value} we have
    \[
        |\mathcal{T}| \ll T^{2\alpha_2} P_2^{2\alpha_2} \exp\left( 2 \frac{\log T}{\log P_2} \log \log T \right) \ll T^{\frac{1}{2} - 4 \beta + 2 \eta} P_2^{2 \alpha_2} \ll T^{\frac{1}{2} - 30 \eta} P_2^{2 \alpha_2}.
    \]
    
    Using upper bound in (\ref{J defn}), for large enough \( C \) in terms of \( \eta \), we have
    \[
        |\mathcal{T}| \sqrt{T} (2 P_3)^J = |\mathcal{T}| \sqrt{T} y^J \leq |\mathcal{T}| \sqrt{T} y X^{\frac{400}{C}} \ll X^{1-\delta}
    \]
    for some \( \delta = \delta(\eta) > 0 \). Thus, we conclude that
    \[
        \int_{\mathcal{T}_3} |F(1+it)|^2 \, dt \ll_{\eta} J^2 P_3^{-2J \sigma_0} \log T \left( \frac{\log X}{\log P_1 \log P_2} \right)^2 \rho(u-v).
    \]
    Finally, observe that for large enough \( X \), we have
    \[
        100 u \log u < \frac{1}{2} J \sigma_0 \log y \quad \text{and} \quad 100 J \log \log X < \frac{1}{2} J \sigma_0 \log y.
    \]
    Thus, we obtain
    \[
        100 (u \log u + J \log \log X) < J \sigma_0 \log y.
    \]
    Since \( P_3 = \frac{y}{2} \), using the above inequality, we get
    \[
        P_3^{-2J \sigma_0} \leq \exp\left( -J \sigma_0 \log y \right) < \left( \frac{\rho(u-v)}{(\log X)^J} \right)^{50}.
    \]
    Therefore, we conclude that
    \[
        \int_{\mathcal{T}_3} |F(1+it)|^2 \, dt \ll_{\eta} \left( \frac{\rho(u-v)}{(\log X)^J} \right)^{25}.
    \]
    This completes the proof.
\end{proof}
\begin{remark}
    It might be possible to reduce the number of \( \log X \) factors that arise due to multiplicity in the counting argument in (\ref{coeff bound T1}) by using a variant of \cite[Lemma 12]{matomaki-radziwill}. However, it seems difficult to control the distribution of prime factors of a \( y \)-smooth number from a given interval for small values of \( y \) due to the lack of good sieving techniques.
\end{remark}
\section{Finishing the Proof of Theorem \ref{almost all interval}}\label{Chebysev style}
In this section, we complete the proof of Theorem \ref{almost all interval}.
\begin{proof}[Proof of Theorem \ref{almost all interval}]
    Set \( \eta := \frac{\epsilon}{64} \) and \( \phi := \frac{13}{8} \). We begin by bounding the right-hand side of the equation (\ref{par app}). For the first term, it follows from (\ref{wt bound}) that 
    \[
        \frac{(J \max_{n} w_n)^2}{y^{1/8}} \ll_{\epsilon} \frac{J^2 (\log X)^6}{\rho(u-v)^{2\epsilon} y^{1/8}}.
    \]
    Using (\ref{y very large}) and Lemma \ref{notational ease} (ii), we can bound the right-hand side above by
    \begin{align}\label{first term bound}
        \ll \left( \frac{\rho(u-v)}{\log X} \right)^{\epsilon} \left( \frac{\rho(u-v)}{\log P_2 (\log P_3)^J} \right)^2.
    \end{align}
    Using Proposition \ref{main propn} and setting \( h_1 = P_1 \rho(u-v)^{1-\phi} \), the remaining terms in \eqref{par app} are 
    \begin{align}\label{upper bound with desired short interval length}
         \ll_{\epsilon} P_1^{-\frac{1}{2}+\epsilon} \rho(u-v)^{\phi-\epsilon} \left( \frac{\log X}{\log P_2 (\log P_3)^J} \right)^2 + \left( \frac{\rho(u-v)}{(\log X)^J} \right)^{25}.
    \end{align}
    Next, if we set 
    \[
    P_1 := \left( \frac{(\log X)^{2+\epsilon}}{\rho(u-v)^{2-\phi+2\epsilon}} \right)^{\frac{2}{1-2\epsilon}},
    \]
    we obtain the following bound for (\ref{upper bound with desired short interval length}):
    \[
    \ll_{\epsilon} \left( \frac{\rho(u-v)}{\log X} \right)^{\epsilon} \left( \frac{\rho(u-v)}{\log P_2 (\log P_3)^J} \right)^2.
    \]
    Combining these results, we conclude that
    \[
    \frac{1}{X} \int_X^{2X} \left| \frac{1}{h_1} S_{h_1}(x) - \frac{1}{h_2} S_{h_2}(x) \right|^2 \, dx \ll_{\epsilon} \left( \frac{\rho(u-v)}{\log X} \right)^{\epsilon} \left( \frac{\rho(u-v)}{\log P_2 (\log P_3)^J} \right)^2.
    \]
    
    Now, let \( \mathcal{E} \) be the set of \( x \in [X, 2X] \) such that 
    \[
    \left| \frac{1}{h_1} S_{h_1}(x) - \frac{1}{h_2} S_{h_2}(x) \right| \geq \frac{\rho(u-v)}{(\log X)^{\epsilon/4} \log P_1 \log P_2 (2 \log P_3)^J},
    \]
    then 
    \[
    \frac{\abs{\mathcal{E}}}{X} \left( \frac{\rho(u-v)}{(\log X)^{\epsilon/4} \log P_1 \log P_2 (2 \log P_3)^J} \right)^2 \ll_\epsilon \left( \frac{\rho(u-v)}{\log X} \right)^{\epsilon} \left( \frac{\rho(u-v)}{\log P_2 (\log P_3)^J} \right)^2.
    \]
    By rearranging the above expression and noting that \( \log P_1 \ll \left( \frac{\log X}{\rho(u-v)} \right)^{\epsilon/8} \) and \( 2^J \ll \rho(u-v)^{-\epsilon/4} \), we obtain
    \[
    |\mathcal{E}| \ll_{\epsilon} X \left( \frac{\rho(u-v)}{\log X} \right)^{\epsilon/4} = o(X).
    \]
    Finally, for \( x \in [X, 2X] \setminus \mathcal{E} \), Lemma \ref{weights av} implies that 
    \[
    \frac{1}{h_1} S_{h_1}(x) \gg \frac{\rho(u-v)}{\log P_1 \log P_2 (2 \log P_3)^J} \gg_\epsilon \left( \frac{\rho(u-v)}{\log X} \right)^{1+\epsilon}.
    \]
    Here, the implicit constant is independent of \( x \). From this, it follows that the interval \( [x, x+h] \) contains a \( y \)-smooth number for \( h \geq h_1 \) whenever \( x \in [X, 2X] \setminus \mathcal{E} \). In particular, we may take the length of the interval as specified in the hypothesis of Theorem \ref{almost all interval} by using Lemma \ref{notational ease} (i) and the fact that \( \epsilon \) is arbitrary. Hence, the proof is complete.        
\end{proof}
\section{Existence in All Intervals}
In this section, we prove Theorem \ref{all interval}. The proof follows a similar approach to that of Theorem \ref{almost all interval}, but we adjust the choice of weights to optimize the $\log$-factors.  

Let \( x \geq 2 \) be large enough, \( \phi := \frac{13}{8} \), and let \( C \) be a sufficiently large but fixed positive constant. The precise choice of \( C \) will become clear in the course of the argument.  

For  
\[
\exp\left( C (\log x)^{2/3} (\log \log x)^{4/3} \right) \leq y \leq x^{\frac{1}{C}},
\]  
define  
\[
    u_0 := \frac{\log x}{2\log y},
\]  
and set  
\[
    P_1 := \left( \frac{(\log x)^{2+2\epsilon}}{\rho(u_0)^{4-2\phi+4\epsilon}} \right)^{\frac{2}{1-4\epsilon}}.
\]

Additionally, we define  
\[
X := \sqrt{x P_1}.
\]
The parameters \( u \), \( J \), \( v \), \( P_2 \), \( P_3 \), and \( \mathcal{M} \) are given by the respective definitions in  
\eqref{u defn}, \eqref{J defn}, \eqref{v defn}, \eqref{prime sizes}, and \eqref{M defn}.

Note that for an appropriately chosen \( C \), the value of \( y \) satisfies \eqref{y range}, and \( P_1 \) meets the condition in \eqref{prime sizes} by applying Lemma \ref{notational ease}.  

Furthermore, define  
\[
    R := \sqrt{\frac{x}{P_1}},
\]  
and set  
\[
    \mathcal{R} := (R, 2R] \cap \mathbf{S}(y).
\]

We define our new set of weights as 
\begin{align}\label{all int: wts}
    \tilde{w}_n := \sum_{\substack{n = q_1 q_2 p_1 \dots p_J mr \\ q_1 \sim P_1, \, q_2 \sim P_2 \\ P_3 < p_1, \dots, p_J \leq 2P_3 \\ m \in \M, \, r \in \mathcal{R}}} 1.
\end{align}

The above weights are supported on $\textbf{S}(y)\cap \left[\frac{x}{2^{J+5}}, 2^{J+6}x\right]$.

Let $G(s)$ be a Dirichlet polynomial with coefficients $\tilde{w}_n$, defined as 
\begin{align}\label{all int:poly defn}
    G(s) := \sum_{n} \frac{\tilde{w}_n}{n^s}.
\end{align}
It immediately follows that 
\begin{align*}
    G(s) = F(s)R(s) = P_1(s) P_2(s) P_3(s)^J M(s) R(s),
\end{align*}
where $F(s), P_1(s)$, $P_2(s)$, $P_3(s)$, and $M(s)$ are as defined in Section \ref{Setting up notation}, and 
\begin{align*}
    R(s) := \sum_{r \in \mathcal{R}} \frac{1}{r^s}.
\end{align*}

Similarly to Lemma \ref{weights av}, we establish a lower bound for a moderately long average of the weights $\tilde{w}_n$.
\begin{lemma}\label{all int: moderately long average}
    Let \( \epsilon > 0 \). There exists a positive constant \( C(\epsilon) \) such that for \( x \geq 2 \) and \( x y^{-5/12} \leq h \leq x \),  

\[
\frac{1}{h} \sum_{x \leq n \leq x + h} \tilde{w}_n \gg_{\eta} \frac{\rho(u_0)^{2+\epsilon/2}}{\log P_1 \log P_2 (2 \log P_3)^J}
\]

holds, provided that \( C \geq C(\epsilon) \).
\end{lemma}

\begin{proof}
    Using (\ref{all int: wts}), we have 
    \begin{align}\label{all int: wts av}
        \frac{1}{h} \sum_{x \leq n \leq x + h} \tilde{w}_n = \frac{1}{h} \sum_{q_1 \sim P_1} \sum_{q_2 \sim P_2} \sum_{P_3 < p_1, \dots, p_J \leq 2P_3} \sum_{r \in \mathcal{R}} \subsum{\frac{x}{q_1 q_2 p_1 \dots p_J r} \leq m \leq \frac{x + h}{q_1 q_2 p_1 \dots p_J r} \\ m \in \textbf{S}(y)} 1.
    \end{align}

    The innermost summation is over smooth numbers in a short interval. From our hypothesis, it follows that 
    \[
        \frac{x}{q_1 q_2 p_1 \dots p_J r} y^{-5/12} \leq \frac{h}{q_1 q_2 p_1 \dots p_J r} \leq \frac{x}{q_1 q_2 p_1 \dots p_J r}.
    \]
    Therefore, we can apply Hildebrand's short interval estimate from Lemma \ref{smooth-number-estimates} (ii). This gives
    \begin{align}\label{all int: inner sum smooth}
        \subsum{\frac{x}{q_1 q_2 p_1 \dots p_J r} \leq m \leq \frac{x + h}{q_1 q_2 p_1 \dots p_J r} \\ m \in \textbf{S}(y)} 1 \gg \frac{h}{q_1 q_2 p_1 \dots p_J r} \rho\pare{\frac{\log \frac{x}{q_1 q_2 p_1 \dots p_J r}}{\log y}}.
    \end{align}

    Combining (\ref{all int: wts av}) and (\ref{all int: inner sum smooth}), we obtain:
    \begin{align*}
        \frac{1}{h} \sum_{x \leq n \leq x + h} \tilde{w}_n & \gg \sum_{q_1 \sim P_1} \sum_{q_2 \sim P_2} \sum_{P_3 < p_1, \dots, p_J \leq 2P_3} \sum_{r \in \mathcal{R}} \frac{1}{q_1 q_2 p_1 \dots p_J r} \rho \left( \frac{\log \frac{x}{q_1 q_2 p_1 \dots p_J r}}{\log y} \right) \\
        & \gg \rho \left( \frac{\log \frac{x}{P_1 P_2 P_3^J R}}{\log y} \right) \rho\pare{\frac{\log (2R)}{\log y}} \frac{1}{\log P_1 \log P_2 (2 \log P_3)^J}.
    \end{align*}

    Finally, applying Lemma \ref{notational ease} (i) and (iii), we conclude that:
    \[
        \frac{1}{h} \sum_{x \leq n \leq x + h} \tilde{w}_n \gg_{\eta} \frac{\rho(u_0)^{2+\epsilon/2}}{\log P_1 \log P_2 (2 \log P_3)^J}.
    \]
\end{proof}
We now prove the following mean-value estimate for \( G(s) \), which is a quick consequence of Proposition \ref{main propn}.

\begin{lemma}\label{all int: main propn}
    Let \( \epsilon > 0 \) be small, and let \( \phi := \frac{13}{8} \). Let \( G(s) \) be as in (\ref{all int:poly defn}). Then, there exist positive constants \( x_0(\eta) \) and \( C(\epsilon, \eta) \) such that for all \( x \geq x_0(\eta) \), any fixed \( C \geq C(\epsilon, \eta) \), and any \( T \geq 1 \), the following bound holds:
    \begin{align*}
        \int_{y^{1/8}}^{T} \left| G(1+it) \right| \, dt 
        \ll_{\epsilon,\eta} & \, P_1^{-\frac{1}{4} + 32\eta} \left( \sqrt{\frac{P_1}{x}}T\rho(u_0)^{1-\epsilon} + \rho(u_0)^{\phi - \epsilon} \right) \frac{\log X}{\log P_2 (\log P_3)^J} \\
        & \quad + \left( \frac{\rho(u_0)}{(\log X)^J} \right)^{10} \pare{\sqrt{\frac{P_1}{x}}T\rho(u_0) + \rho(u_0)^{\phi - \epsilon}}^{1/2}.
    \end{align*}
\end{lemma}
\begin{proof}
     Using the Cauchy-Schwarz inequality, we have 
    \begin{align}\label{after cs}
        \int_{y^{1/8}}^T |G(1+it)| \, dt 
        \leq \pare{\int_{y^{1/8}}^T |F(1+it)|^2 \, dt}^{1/2} \pare{\int_{-T}^T |R(1+it)|^2 \, dt}^{1/2}.
    \end{align}

    By Proposition \ref{main propn}, we estimate the first integral as follows:
    \begin{align}\label{bound for the first integral}
        \int_{y^{1/8}}^T |F(1+it)|^2 \, dt 
        & \ll P_1^{-\frac{1}{2} + 64\eta} \left( \frac{P_1 T}{X} \rho(u-v)^{1-\epsilon} + \rho(u-v)^{\phi - \epsilon} \right) \frac{(\log X)^2}{(\log P_2 (\log P_3)^J)^2} \nonumber\\
        & \quad + \left( \frac{\rho(u-v)}{(\log X)^J} \right)^{25}.
    \end{align}

    For the second integral, we use Lemma \ref{improved mvt} and Lemma \ref{solns-prod-linear} to obtain:
    \begin{align}\label{bound for the second integral}
        \int_{-T}^{T} |R(1+it)|^2 \, dt 
        \ll_{\epsilon} \frac{T}{R} \rho\pare{\frac{\log R}{\log y}} + \rho\pare{\frac{\log R}{\log y}}^{\phi - \epsilon}.
    \end{align}

    Combining \eqref{after cs}, \eqref{bound for the first integral}, and \eqref{bound for the second integral}, and applying Lemma \ref{notational ease} (i) and (iii), we obtain the desired result.
\end{proof}
We will need the following smoothing function in the proof of Theorem \ref{all interval}. Let
$$
\eta_{\xi, \kappa} (z) = \begin{cases}
1 & \text{if } 1 - \kappa \leq z \leq 1 + \kappa, \\
\frac{1 + \kappa + \xi - z}{\xi} & \text{if } 1 + \kappa \leq z \leq 1 + \xi + \kappa, \\
\frac{z + \kappa + \xi - 1}{\xi} & \text{if } 1 - \xi - \kappa \leq z \leq 1 - \kappa, \\
0 & \text{otherwise}.
\end{cases}
$$
We define \( \widetilde{\eta}_{\xi, \kappa} \) as the Mellin transform of \( \eta_{\xi, \kappa} \). That is,
\[
\widetilde{\eta}_{\xi, \kappa}(s) := \int_{0}^{\infty} t^{s-1} \eta_{\xi, \kappa}(t) \, dt.
\]
By Mellin inversion, we have the following representation:
\begin{equation}
\label{eq:Meleta}
    \eta_{\xi, \kappa}(z) = \frac{1}{2\pi i} \int_{1 - i\infty}^{1 + i\infty} z^{-s} \cdot \widetilde{\eta}_{\xi,\kappa}(s) \, ds. 
\end{equation}
Next, observe that
\begin{align*}
    s \cdot \widetilde{\eta}_{\xi, \kappa}(s) &= - \int_{0}^{\infty} t^s \, d\eta_{\xi, \kappa}(t) \\
    &= - \int_{1 - \kappa - \xi}^{1 - \kappa} \frac{t^s}{\xi} \, dt + \int_{1 + \kappa}^{1 + \kappa + \xi} \frac{t^s}{\xi} \, dt \\
    &= \frac{(1 + \xi + \kappa)^{s + 1} - (1 + \kappa)^{s + 1}}{\xi (s + 1)} - \frac{(1 - \kappa)^{s + 1} - (1 - \xi - \kappa)^{s + 1}}{\xi (s + 1)}.
\end{align*}

\begin{proof}[Proof of Theorem \ref{all interval}]
    Set \( \eta := \frac{\epsilon}{32} \). Define \( h_1 \) and \( h_2 \) as follows:  
\[
    h_1 := h_0 \sqrt{x}, \quad \text{where } h_0 := \sqrt{P_1} \rho(u_0)^{1 - \phi},
\]
\[
    h_2 := x y^{-3/8}.
\]

For \( j = 1, 2 \), define \( \kappa_j := \frac{h_j}{x} \), \( \xi_j := \frac{h_j}{x} \), and set \( \eta_j := \eta_{\xi_j, \kappa_j} \). Finally, let
\begin{align}\label{mellin transformed sum}
    S_j := \sum_{n} \tilde{w}_n \eta_j\left( \frac{n}{x} \right)
= \frac{1}{2\pi i} \int_{1 - i \infty}^{1 + i \infty} x^s \cdot G(s) \widetilde{\eta}_j(s) \, ds.
\end{align}
We relate the short-interval distribution to the long-interval distribution by showing that
\begin{align}\label{short long bilinear}
    \left| \frac{S_1}{h_1} - \frac{S_2}{h_2} \right|
\end{align}
is small.

Decompose the integral on the right-hand side of \eqref{mellin transformed sum} as
\begin{align*}
    \int_{1 - i \infty}^{1 + i \infty} x^s \cdot G(s) \widetilde{\eta}_j(s) \, ds &= U_j + V_j,
\end{align*}
where
\begin{align*}
    U_j & := \int_{|t| \leq y^{1/8}} x^s \cdot G(s) \widetilde{\eta}_j(s) \, ds, \\
    V_j & := \int_{|t| > y^{1/8}} x^s \cdot G(s) \widetilde{\eta}_j(s) \, ds.
\end{align*}

By a straightforward modification of the proof of \cite[Theorem 4]{matomaki-radziwill}, incorporating (\ref{polynomial has different bound}), Lemma \ref{notational ease} (iii), and (\ref{first term bound}), we obtain
\begin{align*}
    \left| \frac{U_1}{h_1} - \frac{U_2}{h_2} \right| 
    & \ll \frac{|G(1)| y^{1/4} h_2}{x} \\
    &= \frac{|F(1)R(1)| y^{1/4} h_2}{x} \\
    & \ll J \frac{\max_n |w_n|\rho(u_0)}{y^{1/8}} \\
    & \ll \left( \frac{\rho(u_0)}{\log X} \right)^{\epsilon} \frac{\rho(u_0)^2}{\log P_2 (\log P_3)^J}.
\end{align*}
and
\begin{align*}
    \left| \frac{V_1}{h_1} - \frac{V_2}{h_2} \right| 
    & \ll \sum_{j = 1}^{2} \left( \int_{y^{1/8}}^{x / h_j} |G(1 + it)| \, dt 
    + \frac{x}{h_j} \max_{T > \frac{x}{h_j}} \frac{1}{T} \int_{T}^{2T} |G(1 + it)| \, dt \right).
\end{align*}
    For sufficiently large \( x \) depending on \( \epsilon \), from Lemma \ref{all int: main propn}, we obtain
\begin{align*}
    \left| \frac{V_1}{h_1} - \frac{V_2}{h_2} \right| 
    \ll_{\epsilon} & P_1^{-\frac{1}{4} + \epsilon} \left( \frac{\sqrt{P_1}}{h_0} \rho(u_0)^{1-\epsilon} + \rho(u_0)^{\phi-\epsilon} \right) \frac{\log X}{\log P_2 (\log P_3)^J} \\
    & + \left( \frac{\rho(u_0)}{(\log X)^J} \right)^{10} \left( \frac{\sqrt{P_1}}{h_0} \rho(u_0)^{1-\epsilon} + \rho(u_0)^{\phi-\epsilon} \right)^{1/2}.
\end{align*}
Substituting the expressions for \( P_1 \) and \( h_0 \), we get
\begin{align*}
    \left| \frac{V_1}{h_1} - \frac{V_2}{h_2} \right| 
    & \ll_{\epsilon} P_1^{-\frac{1}{4} + \epsilon} \rho(u_0)^{\phi - \epsilon} \frac{\log X}{\log P_2 (\log P_3)^J} + \left( \frac{\rho(u_0)}{(\log X)^J} \right)^{10} \\
    & \ll_{\epsilon} \left( \frac{\rho(u_0)}{\log X} \right)^{\epsilon} \frac{\rho(u_0)^2}{\log P_2 (\log P_3)^J}.
\end{align*}
Thus, we conclude that
\begin{align}\label{larger t contri}
    \left| \frac{S_1}{h_1} - \frac{S_2}{h_2} \right| 
    \ll_{\epsilon} \left( \frac{\rho(u_0)}{\log X} \right)^{\epsilon} \frac{\rho(u_0)^2}{\log P_2 (\log P_3)^J}.
\end{align}
From this and the definition of \( \eta_2 \), we conclude that
\begin{align}\label{all int: S_1 lower bound}
    \frac{S_1}{h_1} 
    \geq \frac{1}{h_2} \sum_{\substack{x - h_2 \leq n \leq x + h_2}} \tilde{w}_n 
    + O_\epsilon \left( \left( \frac{\rho(u_0)}{\log X} \right)^{\epsilon} \frac{\rho(u_0)^2}{\log P_2 (\log P_3)^J} \right).
\end{align}
By Lemma \ref{all int: moderately long average}, we have
\begin{align*}
    \frac{1}{h_2} \sum_{\substack{x - h_2 \leq n \leq x + h_2}} \tilde{w}_n 
    \gg_\epsilon \frac{\rho(u_0)^{2+\epsilon/2}}{\log P_1 \log P_2 (2 \log P_3)^J}.
\end{align*}
Substituting this into (\ref{all int: S_1 lower bound}) and using the definition of \( \eta_1 \), we obtain
\begin{align*}
    \sum_{\substack{x - 2h_1 \leq n \leq x + 2h_1}} \tilde{w}_n 
    \gg_\epsilon h_1 \frac{\rho(u_0)^{2+\epsilon/2}}{\log P_1 \log P_2 (2 \log P_3)^J}.
\end{align*}
In particular, this shows that every interval of length at least \( 4h_1 \) contains a \( y \)-smooth number. Noting that \( \rho(u_0) = \rho\left(\frac{1}{2} \frac{\log x}{\log y}\right) \), we can take the interval length \( h \) as specified in the hypothesis of Theorem \ref{all interval}, using the fact that \( \epsilon \) is arbitrary. Hence, we conclude the result.
\end{proof}

\end{document}